\documentclass[12pt,reqno]{amsart}
\usepackage{palatino,mathpazo}
\usepackage{amsmath,amssymb,mathrsfs}
\usepackage{xcolor}
\colorlet{mdtRed}{red!50!black}
\definecolor{dblue}{rgb}{0,0,.6}
\usepackage[colorlinks,pagebackref=true]{hyperref}
\hypersetup{colorlinks,linkcolor={red},citecolor={blue},urlcolor={red}}
\usepackage[all]{xy}
\usepackage{tikz,tikz-cd,tkz-graph,enumerate}
\renewcommand*{\backref}[1]{}
\renewcommand*{\backrefalt}[4]{[{%
		\ifcase #1 Not cited.%
		\or $\uparrow$~#2.%
		\else $\uparrow$~#2.%
		\fi%
	}]}

\DeclareMathOperator{\Pic}{\textnormal{Pic}}

\DeclareMathOperator{\Br}{\textnormal{Br}}
\DeclareMathOperator{\Id}{{\rm Id}}

\newcommand{\mf}[1]{\mathfrak{#1}}
\newcommand{\mc}[1]{\mathcal{#1}}
\newcommand{\bb}[1]{\mathbb{#1}}

\newtheorem{theorem}{Theorem}[section] 
\newtheorem{lemma}[theorem]{Lemma} 
\newtheorem{proposition}[theorem]{Proposition}
\newtheorem{corollary}[theorem]{Corollary}

\theoremstyle{definition}
\newtheorem{definition}[theorem]{Definition}
\newtheorem{remark}[theorem]{Remark}

\numberwithin{equation}{section}

\begin{document}
	
	\baselineskip=15.5pt 
	
	\title[Brauer group of stable parabolic $\text{SL}(r,\bb{C})$ and $\text{PGL}(r,\bb{C})$-connections and Higgs]{Brauer group of moduli of stable parabolic $\text{SL}(r,\bb{C})$ and $\text{PGL}(r,\bb{C})$-connections and Higgs bundles over a curve}

	\author[P. Adroja]{Pavan Adroja}
	
	\address{Department of Mathematics, Indian Institute of Technology Gandhinagar, Near Village Palaj, Gandhinagar-382055, Gujarat, India}
	
	\email{pavan.a@alumni.iitgn.ac.in, adrojapavan@gmail.com}
	
	\author[S. Chakraborty]{Sujoy Chakraborty}
	
	\address{Department of Mathematics, 
		Indian Institute of Science Education and Research Tirupati, Andhra Pradesh 517507, India}
	\email{sujoy.cmi@gmail.com}

	\subjclass[2010]{14D20, 14D22, 14F22, 14H60}
	
	\keywords{Brauer Group; moduli space; parabolic bundle; paraboic connection; Higgs bundle}

	\begin{abstract}
		Let $X$ be a compact Riemann surface of genus at least $3$. We compute the Brauer groups of the moduli spaces of stable parabolic $\text{SL}(r,\bb{C})$-connections and stable strongly parabolic $\text{SL}(r,\bb{C})$-Higgs bundles over $X$. We also establish an equality of the Brauer group of the moduli stack of stable parabolic $\text{PGL}(r,\bb{C})$-connections and the smooth locus of its coarse moduli space.
	\end{abstract}
	
	\maketitle
	
	\section{Introduction}
	
	The Brauer group is an important invariant of schemes and algebraic stacks. The cohomological Brauer group of an algebraic variety $Y$ is defined as the torsion subgroup of $H^2_{\mathrm{\acute{e}t}}(Y,\mathbb{G}_m)$; moreover, if $Y$ is smooth, then $H^2_{\mathrm{\acute{e}t}}(Y,\mathbb{G}_m)$ is torsion. The Brauer group of the moduli spaces measures the obstruction to the existence of a universal bundle. It also plays a central role in questions of rationality. In this article, we aim to compute the Brauer groups of certain moduli spaces.
	
	Parabolic vector bundles, introduced by V. B. Mehta and C. S. Seshadri \cite{MS}, refine classical vector bundles by incorporating weighted flags along a divisor. Parabolic connections encode linear differential equations with regular singularities compatible with these filtrations. Via the Hodge moduli space of $\lambda$-connections, introduced by Carlos Simpson, the moduli space of parabolic Higgs bundles arises as the $\lambda=0$ fiber, while the moduli space of parabolic connections appears as the $\lambda=1$ fiber, thereby relating the two spaces through a natural deformation. The Brauer groups of the moduli space of stable parabolic bundles and the moduli space of vector bundles with logarithmic connections are studied in \cite{BBS, BD}. In this article, we extend those results to the moduli space of stable parabolic connections. In a similar spirit, we compute the Brauer group of the moduli space of stable strongly parabolic Higgs bundles. Our main results are as follows. 
	
	Let $X$ be a compact Riemann surface of genus $g$ with $g\geq 3$. Choose a finite subset of points $S=\{x_1, x_2,\cdots,x_n\}$ in $X$. Fix $\xi\in \Pic(X)$. Let $\mathcal{M}_{pc}^{\boldsymbol{m,\alpha}}\left(r,\xi\right)$ denote the moduli space of stable parabolic connections with fixed determinate $\xi$, rank $r$ and parabolic data $(\boldsymbol{m,\alpha})$, where $\boldsymbol{m}$ and $\boldsymbol{\alpha}$ denote the system of multiplicities and weights, respectively (see Definition \ref{def:parabolic-bundles}). Also, suppose $\mc{M}^{\boldsymbol{m},\boldsymbol{\alpha}}_{Higgs}(r,\xi)$ denote the moduli space of strongly parabolic stable Higgs bundles having fixed determinant $\xi$, rank $r$ and parabolic data $(\boldsymbol{m,\alpha})$.
	
	\begin{theorem}[{Theorem \ref{Br of Para Conn} and Theorem \ref{Br of Higgs}}]
		Let $\boldsymbol{\alpha}$ be a generic system of weights as in Definition \ref{def:generic-weights}. Let $m_i^j$ denote the multiplicity of the flag at $x_j$. The Brauer groups of $\mathcal{M}_{pc}^{\boldsymbol{m,\alpha}}\left(r,\xi\right)$ and $\mc{M}^{\boldsymbol{m},\boldsymbol{\alpha}}_{Higgs}(r,\xi)$ are isomorphic to the cyclic group $\mathbb{Z}/\nu\mathbb{Z}$, where $$\nu=\gcd(r,d,m_1^1,m_2^1,\cdots,m_1^{\ell_1},\cdots,m_{n}^{\ell_n})\,.$$
	\end{theorem}
	
	In \cite{BCD23}, it was shown that the Brauer group of the moduli stack of stable parabolic $\text{PGL}(r,\mathbb{C})$-bundles coincides with the Brauer group of the smooth locus of the associated coarse moduli space under the assumptions of full-flag. This result was further generalized in \cite{BC} for arbitrary flags. In this article, we establish an analogous result for the moduli stack of stable parabolic $\text{PGL}(r,\mathbb{C})$-connections.   
	\begin{theorem}[{Theorem \ref{thm:brauer-group-of-moduli-stack}}]
		Let $\boldsymbol{\alpha}$ be a generic system of weights for connections, as in Definition \ref{def:generic-weights-connections}. Let $\mf{N}^{\boldsymbol{m,\alpha}}_{pc}(r,i)$ denote the moduli stack of parabolic stable $\text{PGL}(r,\bb{C})$-connections on $X$ of topological type $i$ and parabolic data $(\boldsymbol{m,\,\alpha})$. Let $N_{pc}^{\boldsymbol{m,\alpha}}(r,i)$ denote the corresponding coarse moduli space. Let $N_{pc}^{\boldsymbol{m,\alpha}}(r,i)^{sm}$ denote its smooth locus. We have
		$$\textnormal{Br}\left(\mf{N}_{pc}^{\boldsymbol{m,\alpha}}(r,i)\right)\,\ \simeq\,\
		\textnormal{Br}\left(N_{pc}^{\boldsymbol{m,\alpha}}(r,i)^{sm}\right)\,.$$
	\end{theorem}
	
	\section{Preliminaries}\label{sec:preliminaries}
	
	Let $X$ be a compact Riemann surface of genus $g \geq 3$.  Let $\Omega^1_X$ denote the sheaf of holomorphic $1$-forms on $X$. Fix a finite subset of points $S$ in $X$, sometimes referred to as the \emph{set of parabolic points}.
	
	\subsection{Logarithmic connections and residues}\hfill\\ Let $E$ be a holomorphic vector bundle over $X$.
	A \textit{logarithmic connection on $E$ singular over $S$} is a $\mathbb{C}$-linear map
	$$\nabla:E\longrightarrow E\otimes \Omega_X^1(S)=E \otimes \Omega_X^1 \otimes \mathcal{O}_X(S)$$
	satisfying the Leibniz identity
	$\nabla(fs)=f\nabla(s)+df \otimes s$
	for locally defined holomorphic functions $f$ on $X$ and local sections $s$ of $E$. We shall sometimes denote a logarithmic connection as a pair $(E,\nabla)$.
	
	Two logarithmic connections $(E,\nabla)$ and $(E',\nabla')$ having singularities along $S$ are said to be \textit{isomorphic}, if there exists an $\mc{O}_X$-linear isomorphism $\phi: E\xrightarrow{\,\,\simeq\,\,} E'$ making the following diagram commute:
	\begin{align}
		\xymatrix{E \ar[rr]^(.4){\nabla} \ar[d]_{\phi} && E\otimes \Omega^1_X(S) \ar[d]^{\phi\otimes\Id} \\
			E'\ar[rr]^(.4){\nabla'} && E'\otimes\Omega^1_X(S)
		}
	\end{align} 
	
	We now come to the notion of the residue of a logarithmic connection. For each point $x\in S$, the fiber $\Omega_X^1(S)_x$ can be canonically identified with $\bb{C}$ using Poincar\'e adjunction formula. To be more explicit, if $z$ is a local coordinate around $x$ with $z(x)=0$, any element
	$\omega \in \Omega_X^1(S)_x$ can be written uniquely as
	\[
	\omega = a\,\frac{dz}{z} + \eta
	\]
	where $a \in \mathbb{C}$ and $\eta$ is a holomorphic 1-form at $x$.
	The adjunction map is given by sending $\omega$ to $a$.
	This description is independent of the choice of the local coordinate $z$. For a logarithmic connection $(E,\nabla)$,  the composition
	\[E\xrightarrow{\,\,\,\nabla\,\,\,}E\otimes\Omega_X^1(S)\longrightarrow \left(E\otimes\Omega^1_X(S)\right)_x \simeq E_x\]
	is $\mc{O}_X$-linear. This gives rise to a $\bb{C}$-linear endomorphism of $E_x$, known as the \textit{residue} of $\nabla$ at $x$. It is denoted by $\text{Res}_x(\nabla)$.
	
	
	\subsection{parabolic vector bundles and parabolic connections}
	\begin{definition}\label{def:parabolic-bundles}
		A \textit{parabolic vector bundle} of rank $r$ on $X$ is a holomorphic vector bundle $E$ of
		rank $r$ together with the data of a weighted flag on the fiber of $E$ over each $x\,\in\, S$:
		\begin{align}\label{eqn:parabolic-data}
			E_x =: E^x_1&\, \supsetneq\, E^x_2\, \supsetneq\, \cdots\,\supsetneq\, E^x_{\ell(x)}\, \supsetneq\,  E^x_{\ell(x)+1}\,=\,0\\ 
			0\,\leq\,& \alpha^x_{1}\,<\,\alpha^x_{2}
			\,<\,\cdots\,<\,\alpha^{x}_{\ell(x)}\,<\,1.\nonumber
		\end{align}
		\begin{enumerate}[$\bullet$]
			\item Such a flag is said to be of length $\ell(x)$, and the numbers $m^x_{j}
			\,:=\, \dim E^x_{j} -
			\dim E^{x}_{j+1}$ is said to be the \textit{multiplicity} of the weight $\alpha^x_{j}$.  The collection $\boldsymbol{m}:=\{(m^x_1,m^x_2,\cdots,m^x_{\ell(x)})\}_{x\in S}$ is known as a  \textit{system of multiplicities}. Clearly, we have $\sum_{j=1}^{\ell(x)}m^x_j =r$\,.
			
			\item The collection of real numbers $\boldsymbol{\alpha}\,:=\,\{(\alpha^x_{1}
			\,<\,\alpha^x_{2}\,<\,\cdots\,<\,\alpha^x_{\ell(x)})\}_{x\in S}$ is called a \textit{system
				of weights}.  The tuple $(\boldsymbol{m,\alpha})$ is known as a \textit{parabolic data along $S$}.
		\end{enumerate}
		We shall often denote a parabolic vector bundle simply by $E_*$ and suppress the parabolic data, when there is no scope of confusion. When $E$ is of rank $r$ and degree $d$, The \emph{parabolic degree} of $E_{*}$ is defined as the real number
		$$p\mathrm{deg}(E_{*})=d + \sum_{x\in S}\sum_{j=1}^{\ell(x)} m_j^x\alpha_j^x\,.$$
		The \emph{parabolic slope} of $E_{*}$ is defined as the real number
		$p\mu(E_{*}):=\frac{p\mathrm{deg}(E_{*})}{r}\,.$
	\end{definition}
	
	Let $E_{*}$ be a parabolic bundle and $F$ be a vector subbundle of $E$. Then the parabolic structure on $E$ induces a parabolic structure on $F$ in a natural way \cite{BY,MS}. 
	The vector bundle $F$ together with this induced parabolic structure is denoted by $F_{*}$ and is called a parabolic subbundle of $E$.
	
	\begin{definition}
		A parabolic vector bundle $E_{*}$ is called \emph{parabolic semistable} (respectively \emph{parabolic stable}) if for every non-zero proper parabolic subbundle $F_{*}$, we have 
		$$p\mu(F_{*}) \leq p\mu(E_*) \quad (\text{respectively,}\,\,\, p\mu(F_{*}) < p\mu(E_*))\,.$$
	\end{definition}
	\begin{definition}\label{def:parabolic-connections}
		Let $E_*$ be a parabolic vector bundle on $X$ with parabolic data $(\boldsymbol{m,\alpha})$ along $S$. A \textit{parabolic connection on} $E_*$ is a logarithmic connection $\nabla$ on $E$ singular over $S$, such that for all $x\in S$,
		\begin{enumerate}[(i)]
			\item The residue $\text{Res}_x(\nabla)$ is semisimple and preserves the flag at $E_x$ coming from the parabolic structure, i.e. $\text{Res}_x(\nabla)\left(E^x_j\right)\subset E^x_{j+1}$ for all $j$ (Definition \ref{def:parabolic-bundles}), and
			\item $\text{Res}_x(\nabla)\left(E^x_j/E^x_{j+1}\right) = \alpha^x_j\cdot\text{Id}_{_{E^x_j/E^x_{j+1}}}$ for all $j$.
		\end{enumerate}
	\end{definition}
	\begin{definition}
		A parabolic connection $(E_*,\nabla)$ is said to be \textit{parabolic semistable} (respectively \textit{parabolic stable}), if for any proper sub-bundle $F$ preserved under $\nabla$, we have 
		$$p\mu(F_{*}) \leq p\mu(E_*) \quad (\text{respectively,}\,\,\, p\mu(F_{*}) < p\mu(E_*))\,.$$
	\end{definition}
	\subsection{Generic systems of weights}\hfill\\
	Fix a system of multiplicities $\boldsymbol{m}$ of rank $r$ along $S$. In other words, for any vector bundle $E$ of rank $r$ over $X$, the data $\boldsymbol{m}$ specifies the flag-type at the fibers $E_x$ over for each point in $x\in S$. There is an obvious notion of when a system of weights $\boldsymbol{\alpha}$ is compatible with this given system of multiplicities; namely, at each point in $S$, the number of weights should match with the number of multiplicities at that point. 
	
	\begin{definition}[\text{\cite[\S~2]{BY}}]\label{def:generic-weights}
		Fix a system of multiplicities $\boldsymbol{m}$ of rank $r$ along $S$. A system of weights $\boldsymbol{\alpha}$ is known as \textit{generic}, if parabolic semistability and parabolic stability coincide for any parabolic vector bundle of rank $r$ and parabolic data $(\boldsymbol{m,\alpha})$.
	\end{definition} 
	One can similarly talk about a generic system of weights for parabolic connections.
	\begin{definition}[\text{\cite[p. 575]{BY96}}]\label{def:generic-weights-connections}
		Fix a system of multiplicities $\boldsymbol{m}$ of rank $r$ along $S$. A system of weights $\boldsymbol{\alpha}$ compatible with $\boldsymbol{m}$ is said to be \textit{generic for parabolic connections}, if if parabolic semistability and parabolic stability coincide for any parabolic connection $(E_*,\nabla)$ singular over $S$ of rank $r$ and having parabolic data $(\boldsymbol{m,\alpha})$.
	\end{definition}
	\section{Moduli space of (twisted) stable parabolic $\text{SL}(r,\bb{C})-$connections}\label{section:moduli-of-connections}
	Fix an integer $d$, and consider a parabolic data $(\boldsymbol{m,\alpha})$ along $S=\{x_1,x_2,\cdots,x_n\}$ satisfying
	\begin{align}\label{eqn:degree-condition}
		d+\sum_{i=1}^{n}\sum_{j=1}^{\ell_i}m^i_j\alpha^i_j =0\,.
	\end{align}
	Under the additional assumption $\alpha_i^1>0$ for all $1\leq i\leq n$, Inaba and Saito in \cite{IS} constructed the \textit{moduli space of stable parabolic connections }$(E_*,\nabla)$ of rank $r$ and degree $d$ on $X$ along points in $S$, denoted by
	\[\mc{M}^{\boldsymbol{m},\boldsymbol{\alpha}}_{pc}(r,d)\,.\]
	It is a smooth quasi-projective variety of dimension (see \cite[Theorem 1.3]{IS})
	\begin{align}
		\dim \mc{M}^{\boldsymbol{m},\boldsymbol{\alpha}}_{pc}(r,d) = 2r^2(g-1)+2+2\sum_{i=1}^{n}\sum_{j=1}^{\ell_i}\sum_{j'>j} m^i_jm^i_{j'}
	\end{align}
	where $\ell_i := \ell(x_i)$ (see Definition \ref{def:parabolic-bundles}). A parabolic connection $\nabla$ on a parabolic vector bundle $E_*$ with parabolic data $(\boldsymbol{m,\alpha})$ gives rise to a logarithmic connection $\det(\nabla)$ on $\det(E)$ with residues
	$$\text{Res}_{x_i}(\det(\nabla)) = \text{tr}(\text{Res}_{x_i}\nabla)=\sum_{j=1}^{\ell_i}m^i_j\alpha^i_j$$ for each $x_i\in S.$  Denote $\lambda_i := \sum_{j=1}^{\ell_i}m^i_j\alpha^i_j$ and $\boldsymbol{\lambda}:=(\lambda_1,\cdots,\lambda_n)$, and let $\mc{M}_{lc}(1,d,\boldsymbol{\lambda})$ denote the moduli space of line bundles on $X$ of degree $d$ with logarithmic connections along $S$ with residues $\lambda_i$ at $x_i\in S$. This moduli space is nonempty due to the condition \eqref{eqn:degree-condition} \cite[Proposition 4.1]{BH}. Consider the morphism
	\begin{align}\label{eqn:map}
		f: \mc{M}^{\boldsymbol{m},\boldsymbol{\alpha}}_{pc}(r,d) &\xrightarrow{\,\,\,\,\,\,\,\,\,\,\qquad} \mc{M}_{lc}(1,d,\boldsymbol{\lambda})\\
		(E_*,\nabla) &\mapsto (\det(E),\det(\nabla))\nonumber
	\end{align}
	It can be shown that $f$ is a smooth morphism (see \cite[\S~5]{I} for the full-flag case; see also \cite[Theorem 1.2]{IS}). 
	\begin{definition}[\text{\cite{IS}}]\label{def:parabolic-connections-moduli}
		Assume that the condition \eqref{eqn:degree-condition} is satisfied. Fix a line bundle $\xi$ on $X$ of degree $d$, and a logarithmic connection $\nabla_{\xi}$ on $\xi$ with residues $\lambda_i := \sum_{j=1}^{\ell_i}m^i_j\alpha^i_j$ along the points $x_i\in S$ $(1\leq i\leq n)$. We shall denote by 
		$$\mc{M}^{\boldsymbol{m},\boldsymbol{\alpha}}_{pc}(r,\xi)$$
		the \textit{moduli space of stable parabolic connections} $(E_*,\nabla)$ of rank $r$, parabolic data $(\boldsymbol{m,\alpha})$ and determinant $\xi$, such that $(\det(E),\det(\nabla))\simeq (\xi,\nabla_{\xi})$ as logarithmic connections. In other words, $\mc{M}^{\boldsymbol{m},\boldsymbol{\alpha}}_{pc}(r,\xi)$ is the the fiber of the smooth map $f$ \eqref{eqn:map} over a point $(\xi,\nabla_{\xi})\in\mc{M}_{lc}(1,d,\boldsymbol{\lambda})$.
	\end{definition}
	To compute the dimension of $\mc{M}^{\boldsymbol{m},\boldsymbol{\alpha}}_{pc}(r,\xi)$, take any holomorphic line bundle $\xi'$ on $X$ of degree $d$. The condition \eqref{eqn:degree-condition} implies that $\xi'$ admits a logarithmic connection with residues $\lambda_i$ at each $x_i\in S$ \cite[Theorem 1.1]{BL}. Consequently, $\mc{M}_{lc}(1,d,\boldsymbol{\lambda})$ is an affine bundle on $\Pic^d(X)$ whose fiber over $\xi'$ is of dimension
	\[\dim H^0\left(X,\,End(\xi')\otimes\Omega_X^1\right) = \dim H^0(X,\Omega_X^1) = g.\]
	It follows that $\dim \mc{M}_{lc}(1,d,\boldsymbol{\lambda}) = 2g$. Consequently, $\mathcal{M}_{pc}^{\boldsymbol{m,\alpha}}\left(r,\xi\right)$ is a smooth irreducible quasi-projective variety of dimension
	\begin{align}\label{eqn:dimension-parabolic-moduli}
		\dim \mathcal{M}_{pc}^{\boldsymbol{m,\alpha}}\left(r,\xi\right) = \dim \mathcal{M}_{pc}^{\boldsymbol{m,\alpha}}\left(r,d\right) - 2g = 2(r^2-1)(g-1)+2\sum_{i=1}^{n}\sum_{j=1}^{\ell_i}\sum_{j'>j} m^i_jm^i_{j'}
	\end{align} 
	
	\section{Brauer group of $\mc{M}^{\boldsymbol{m},\boldsymbol{\alpha}}_{pc}(r,\xi)$}\label{sec:brauer-group-parabolic-connections}
	\begin{lemma}\label{L:dim-conn}
		Consider a parabolic data $(\boldsymbol{m,\alpha})$ along $S=\{x_1,\cdots,x_n\}$, where $\boldsymbol{\alpha}$ is a generic system of weights in the sense of Definition \ref{def:generic-weights}. Suppose $E$ is a vector bundle of rank $r$ and determinant $\xi$, and assume $E_*$ admits a connection $\nabla$ such that $(E_*,\nabla)$ is stable; in other words, $(E_*,\nabla)\in \mathcal{M}^{\boldsymbol{m,\alpha}}_{pc}(r,\xi)$. Then, the space isomorphism classes of parabolic connections $D$ on $E_*$ such that $(E_*, D)\in \mathcal{M}^{\boldsymbol{m,\alpha}}_{pc}(r,\xi)$ is of  dimension at most $$(r^2-1)(g-1)+\sum_{i=1}^{n}\sum_{j=1}^{\ell_i}\sum_{j'>j} m^i_jm^i_{j'}$$ 
	\end{lemma}
	\begin{proof}
		Let $\text{PC}_{\xi}(E_*)$ denote the space of all parabolic connections on $E_*$ which also satisfy $(\det(E),\det(D))\simeq(\xi,\nabla_{\xi})$ as logarithmic connections. It contains an open subspace $\text{PC}^s_{\xi}(E_*)$ consisting of all parabolic connections $D$ on $E_*$ such that $(E_*,D)$ is \textit{stable}, as well as  $(\det(E),\det(D))\simeq(\xi,D_{\xi})$. Note that $\text{PC}_{\xi}(E_*)$ is a subspace of an affine space modelled over the vector space $\mathrm{H}^0(X, \Omega_X^1(S) \otimes \mathrm{SParEnd}'(E_*))$, where $\mathrm{SParEnd}'(E_*)$ denotes the strongly parabolic morphisms whose trace is zero.
		
		Given $\Phi \in \mathrm{Aut}(E_*)$  and a parabolic connection $D\in \text{PC}^s_{\xi}(E_*)$ on $E_*$, the $\mathbb{C}$-linear morphism $(\Phi \otimes \mathrm{id}_{\Omega^1_X(S)}) \circ D \circ \Phi^{-1}$ again defines a parabolic connection on $E_*$. In fact,
		$$(D,\Phi) \mapsto (\Phi \otimes \mathrm{id}_{\Omega^1_X(S)}) \circ D \circ \Phi^{-1}$$
		defines an action of $\mathrm{Aut}(E_*)$ on $\text{PC}^s_{\xi}(E_*)$. The quotient  $\text{PC}^s_{\xi}(E_*)/\mathrm{Aut}(E_*)$ parametrizes all isomorphic parabolic connections on $E_*$, which is our required space. In a similar manner, we have an action of $\mathrm{Aut}(E_*)$ on $\text{PC}_{\xi}(E_*)$. Clearly, the space $\text{PC}^s_{\xi}(E_*)/\mathrm{Aut}(E_*)$ is an open subset of $\text{PC}_{\xi}(E_*)/\mathrm{Aut}(E_*)$, and hence 
		$$\mathrm{dim}\left(\text{PC}^s_{\xi}(E_*)/\mathrm{Aut}(E_*)\right) = \mathrm{dim}\left(\text{PC}_{\xi}(E_*)/\mathrm{Aut}(E_*)\right) .$$ 
		
		Choose any $D\in \text{PC}^s_{\xi}(E_*)$, and let 
		$$\mathrm{Aut}(E_*)_D=\{ \Phi \in \mathrm{Aut}(E_*)\; |\; (\Phi \otimes \mathrm{id}_{\Omega^1_X(S)}) \circ D \circ \Phi^{-1} = D \}$$
		denote the isotropy subgroup. Since the pair $(E_*,D)$ is stable, 
		it is a standard fact that $\mathrm{Aut}(E_*)_D$ only consists of the scalar automorphism of $E_*$. Hence, 
		$\mathrm{dim}\left(\mathrm{Aut}(E_*)_D\right)=1.$
		
		\noindent
		Since the Lie algebra of $\mathrm{Aut}(E_*)$ is $\mathrm{H}^0(X, \mathrm{ParEnd}(E_*))$, we have
		$$\mathrm{dim}\; \mathrm{Aut}(E_*)=\mathrm{dim}\; \mathrm{H}^0(X, \mathrm{ParEnd}(E_*))=\mathrm{dim}\; \mathrm{H}^0(X, \mathrm{ParEnd}'(E_*))+1,$$
		where $\mathrm{ParEnd}'(E_*)$ denote the space of parabolic endomorphisms of trace zero. Using Parabolic Serre duality $$\mathrm{H}^1(X, \mathrm{ParEnd}'(E_*))\simeq \mathrm{H}^0(X, \Omega_X^1(S) \otimes \mathrm{SParEnd}'(E_*))^{\vee}$$ it follows that the dimension of $\text{PC}^s_{\xi}(E_*)/\mathrm{Aut}(E_*)$ is at most
		\begin{align}
			&\mathrm{dim}\; \mathrm{H}^0(X, \Omega_X^1(S) \otimes \mathrm{SParEnd}'(E_*)) - \mathrm{dim}\; \mathrm{Aut}(E_*) + \mathrm{dim}\; \mathrm{Aut}(E_*)_D\nonumber\\
			=&\,\, \mathrm{dim}\; \mathrm{H}^1(X, \mathrm{ParEnd}'(E_*)) - \mathrm{H}^0(X, \mathrm{ParEnd}'(E_*))- 1 + 1 \nonumber\\
			=& \,\, - \chi (\mathrm{ParEnd}'(E_*)), \label{dim Conn/Aut}
		\end{align}
		where $\chi (\mathrm{ParEnd}'(E_*))$ denotes the Euler-Poincar\'e characteristic of $\mathrm{ParEnd}'(E_*)$ over $X$. 
		Next, consider the following short exact sequence of sheaves on $X$ \cite[Lemma 2.4]{BH}:
		$$0 \longrightarrow \mathrm{ParEnd}'(E_*) \longrightarrow \mathrm{End}'(E) \longrightarrow \mathcal{K}_S \longrightarrow 0$$
		where $\mathrm{End}'(E)$ denotes the space of trace zero endomorphisms of the underlying bundle $E$, and $\mathcal{K}_S$ denotes the natural skyscraper sheaf supported on the parabolic points $S$. From \cite[Lemma 2.4]{BH}, we also get 
		\begin{equation}\label{EulerChar-skyscrapper}
			\chi(\mathcal{K}_S)=\sum_{i=1}^{n}\sum_{j=1}^{\ell_i}\sum_{j'>j} m^i_jm^i_{j'}.  
		\end{equation}
		Then, we have 
		\begin{equation}\label{EulerChar-seq}
			\chi(\mathrm{End}'(E))=\chi(\mathrm{ParEnd}'(E_*))+\chi(\mathcal{K}_S).  
		\end{equation}
		
		Consider also the exact sequence of sheaves over $X$ given by the trace map
		$$0 \longrightarrow \mathrm{End}'(E) \longrightarrow \mathrm{End}(E) \xrightarrow{\,\,\,tr\,\,\,} \mathcal{O}_X \longrightarrow 0\,. $$
		Using Riemann-Roch theorem, we have $\chi(\mathrm{End}(E))=r^2(1-g).$ It follows that
		\begin{equation}\label{EulerChar-End'}
			\chi(\mathrm{End}'(E))= \chi(\text{End}(E))-\chi(\mathcal{O}_X) = (r^2-1)(1-g).  
		\end{equation}
		Therefore, from equations (\ref{dim Conn/Aut}), (\ref{EulerChar-seq}), (\ref{EulerChar-skyscrapper}) and (\ref{EulerChar-End'}), it follows that the dimension of  $\text{PC}^s_{\xi}(E_*)/\mathrm{Aut}(E_*)$ is at most
		$$(r^2-1)(g-1)+\sum_{i=1}^{n}\sum_{j=1}^{\ell_i}\sum_{j'>j} m^i_jm^i_{j'}.$$ 
		This proves our claim.
	\end{proof}
	Let 
	$\mc{U}$
	denote the open subset of $\mathcal{M}^{\boldsymbol{m,\alpha}}_{pc}(r,\xi)$ consisting of the those parabolic connections $(E_*,\nabla)$ whose underlying parabolic vector bundle $E_*$ is stable (see \cite[Theorem 2.8(A)]{M}). Consider its complement
	\begin{equation}\label{Z}
		Z:= \mathcal{M}^{\boldsymbol{m,\alpha}}_{pc}(r,\xi) \setminus \mc{U}. 
	\end{equation}
	We want to compute the codimension of the closed subset $Z$ inside $\mathcal{M}_{pc}^{\boldsymbol{m,\alpha}}(r,\xi)$.
	\begin{lemma}\label{Codim Para}
		Let $Z$ be as mentioned in \eqref{Z}. Then, for $r\geq 2$, and $g\geq 3$, we have
		$$\mathrm{codim}\left(Z, \mathcal{M}^{\boldsymbol{m,\alpha}}_{pc}(r,\xi)\right)\geq 2.$$
	\end{lemma}
	\begin{proof}
		Let $(E_{*}, D)\in Z$. Then, $E_{*}$ is not parabolic stable. Since $\boldsymbol{\alpha}$ is a generic system of weights, $E_{*}$ is not parabolic semi-stable. Let
		$$0=E_{*}^{1} \subset E_{*}^{2} \subset \cdots \subset E_{*}^{l-1} \subset E_{*}^{l}= E_{*} $$
		be the Harder-Narasimhan filtration of the parabolic vector bundle $E_{*}$. The collection of pairs of integers $\{ (\mathrm{rank}(E_{*}^{i}), p\mathrm{deg}(E_{*}^{i}))\}_{i=1}^{l}$ is called the Harder-Narasimhan polygon of $E_{*}$.
		
		The space of all isomorphism classes of parabolic vector bundles over $X$ whose Harder-Narasimhan polygon coincides with the given parabolic vector bundle $E_{*}$ is of dimension at most
		\begin{align}\label{eqn:1}
			r^2(g-1)-(r-1)(g-2)-g+\sum_{i=1}^{n}\sum_{j=1}^{\ell_i}\sum_{j'>j} m^i_jm^i_{j'}.
		\end{align}
		This follows from the analogous techniques of \cite[Lemma 3.1, page no. 303]{BM}.
		
		By Lemma \ref{L:dim-conn}, the dimension of the space of all isomorphism classes of parabolic connections on any given parabolic vector bundle $E_{*}$ is at most
		\begin{align}\label{eqn:2}
			(r^2-1)(g-1)+\sum_{i=1}^{n}\sum_{j=1}^{\ell_i}\sum_{j'>j} m^i_jm^i_{j'}.
		\end{align}
		
		It follows that the dimension $Z$ is at most the sum of \eqref{eqn:1} and \eqref{eqn:2}, namely
		$$\dim(Z)\leq r^2(g-1)-(r-1)(g-2)-g+ (r^2-1)(g-1)+2\sum_{i=1}^{n}\sum_{j=1}^{\ell_i}\sum_{j'>j} m^i_jm^i_{j'}\,.$$
		
		Since we have (cf. \eqref{eqn:dimension-parabolic-moduli})
		$$\mathrm{dim} \left(\mathcal{M}^{\boldsymbol{m,\alpha}}_{pc}(r,\xi)\right)= 2(r^2-1)(g-1)+2\sum_{i=1}^{n}\sum_{j=1}^{\ell_i}\sum_{j'>j} m^i_jm^i_{j'}\,,$$ we get
		\begin{align}
			\mathrm{dim}\left(\mathcal{M}^{\boldsymbol{m,\alpha}}_{pc}(r,\xi)\right)-\dim Z &\geq (r^2-1)(g-1)-r^2(g-1)+(r-1)(g-2)+g \nonumber \\
			&= (r-1)(g-2) + 1\,.
		\end{align}
		Our claim thus follows.
	\end{proof}
	
	\begin{lemma}\label{lem:torsor}
		Let $\mathcal{N}^{\boldsymbol{m,\alpha}}\left(r,\xi\right)$ denotes the moduli space of stable parabolic vector bundles on $X$ of rank $r$, determinant $\xi$ and parabolic data $(\boldsymbol{m,\alpha})$. Consider the forgetful morphism
		\begin{align}\label{eqn:forgetful-morphism}
			\pi: \,&\mc{U}\,\xrightarrow{\qquad\quad}\, \mathcal{N}^{\boldsymbol{m,\alpha}}\left(r,\xi\right)\\
			&(E_*,\nabla)\mapsto E_*\nonumber
		\end{align}
		Then $\mc{U}$ is a $T^*\mathcal{N}^{\boldsymbol{m,\alpha}}\left(r,\xi\right)$-torsor over $\mathcal{N}^{\boldsymbol{m,\alpha}}\left(r,\xi\right)$ via $\pi$.
	\end{lemma}
	\begin{proof}
		The morphism $\pi$ is surjective by \cite[Theorem 3.1]{B}. Moreover, $\pi^{-1}(E_*)$ is an affine space modelled over $\mathrm{H}^0(X, \Omega_X^1(S) \otimes \mathrm{SParEnd}'(E_*))$, where $\mathrm{SParEnd}'(E_*)$ denotes the strongly parabolic morphisms of $E_*$ whose trace is zero. This proves our claim.
	\end{proof}
	\noindent
	\medskip
	We are now in a position to compute the Brauer group for the moduli space $\mathcal{M}_{pc}^{\boldsymbol{m,\alpha}}\left(r,\xi\right)$.
	\begin{theorem}\label{Br of Para Conn}
		If $\boldsymbol{\alpha}$ is a generic system of weights as in Definition \ref{def:generic-weights}, then $\mathrm{Br}\left(\mc{M}_{pc}^{\boldsymbol{m,\alpha}}(r,\xi)\right)$ is isomorphic to the cyclic group $\bb{Z}/\nu\bb{Z}$, where 
		\begin{align}\label{eqn:gcd}
			\nu= \gcd(r,d,m_1^1,m_2^1,\cdots,m_1^{\ell_1},\cdots,m_{n}^{\ell_n})\,.
		\end{align}
	\end{theorem}
	\begin{proof}
		Let us denote $Y:= \mathcal{N}^{\boldsymbol{m,\alpha}}\left(r,\xi\right)$ for convenience. Due to Lemma \ref{lem:torsor}, the map $\pi: \mc{U}\,\longrightarrow\, Y$ is a torsor for $T^*Y$, and thus $\mc{U}$ is an affine bundle over $Y$. It follows that the sheaf of multiplicative groups satisfy
		\begin{enumerate}[$\bullet$]
			\item $\pi_*\bb{G}_{m,\mc{U}} =\bb{G}_{m,Y}$, and
			\item $R^i\pi_*\bb{G}_{m,\mc{U}} =0$ for $i=1,2$.
		\end{enumerate}
		It follows that the five-term exact sequence associated to the spectral sequence
		\[E_2^{p,q}=H^p_{\text{\'et}}\left(Y,R^q\pi_*\bb{G}_{m,\mc{U}}\right)\,\implies\,H^{p+q}_{\text{\'et}}\left(\mc{U},\bb{G}_{m,\mc{U}}\right) \]
		takes the form 
		{\small \begin{align*}
				0\rightarrow \Pic\left(Y\right)\rightarrow\Pic(\mc{U})\rightarrow H^0_{\text{\'et}}(Y,R^1\pi_*\bb{G}_m) =0 \rightarrow H^2_{\text{\'et}}\left(Y,\pi_*\bb{G}_m\right)=\mathrm{Br}\left(Y\right)\rightarrow\mathrm{Br}(\mc{U})\rightarrow 0.
		\end{align*}}
		This immediately implies that $\Pic(\mc{U})\simeq \Pic(Y)$ and $\mathrm{Br}\left(\mc{U}\right)\simeq \mathrm{Br}(Y)$. On the other hand, Lemma \ref{Codim Para} and cohomological purity for Brauer groups \cite{C} implies
		$$\mathrm{Br}(\mc{U}) \simeq \mathrm{Br}(\mathcal{M}^{\boldsymbol{m,\alpha}}_{pc}(r,\xi)) \,.$$
		Thus, $\mathrm{Br}\left(\mathcal{M}^{\boldsymbol{m,\alpha}}_{pc}(r,\xi)\right)\simeq\mathrm{Br}(Y)$. By \cite[Theorem 1.1]{BD}, we know that $\mathrm{Br}(Y)\simeq \bb{Z}/\nu\bb{Z}\,$. This completes the proof.
	\end{proof}
	\section{Brauer group of the moduli of stable parabolic $\text{SL}(r,\bb{C})$-Higgs bundles}
	
	We now consider the case of parabolic Higgs bundles over $X$. As before, fix a parabolic data $(\boldsymbol{m,\alpha})$ of rank $r$ along a set of points $S$ in $X$. 
	\begin{definition}
		A \emph{strongly parabolic Higgs bundle} is a pair $(E_*, \Phi)$, where $E_*$ is a parabolic vector bundle together with an $\mathcal{O}_X$-linear homomorphism
		$$\Phi: E \longrightarrow E \otimes \Omega_X^1(S)$$
		such that for each $x\in S$ the map $\Phi_x$ on the fiber $E_x$ satisfies (see Definition \ref{def:parabolic-bundles})
		$$\Phi(E^x_j)\subset E^x_{j+1}\otimes \Omega_X^1(S)|_{x}$$
		for all $j$. The map $\Phi$ is called a \emph{parabolic Higgs field} on $E_*$. 
	\end{definition}
	
	A parabolic Higgs bundle $(E_*,\Phi)$ is said to be \emph{parabolic semistable} (respectively, \emph{parabolic stable}) if for every non-zero proper parabolic subbundle $F_*$ of $E_*$, which is invariant under $\Phi$, that is, $\Phi(F) \subset F \otimes \Omega_X^1(S)$, one has
	$$p\mu(F_{*}) \leq p\mu(E_*) \quad (\text{respectively,}\,\,\, p\mu(F_{*}) < p\mu(E_*))\,.$$
	
	Fix a holomorphic line bundle $\xi$ on $X$ of degree $d$. Let $\mathcal{M}^{\boldsymbol{m,\alpha}}_{Higgs}(r, \xi)$ denote the moduli space of stable strongly parabolic Higgs bundles $(E_*,\Phi)$ having parabolic data $(\boldsymbol{m,\alpha})$ along $S$, with fixed determinant $\det E \simeq \xi$ and trace$(\Phi)=0$. (see \cite{GL, Y}). It is an irreducible normal quasi-projective variety. 
	
	Let $S=\{x_1,\cdots,x_n\}$. Using the dimension expression in \cite[p. 575]{BY96} and the identity 
	\[
	\sum_{i=1}^{n}
	\left(
	r^2 - \sum_{j=1}^{\ell_i} (m^i_j)^2
	\right)
	=
	2 \sum_{i=1}^{n}
	\sum_{j=1}^{\ell_i}
	\sum_{j' > j}
	m^i_j \, m^i_{j'}\,,
	\]
	it follows that
	\begin{align}
		\dim \mc{M}^{\boldsymbol{m},\boldsymbol{\alpha}}_{Higgs}(r,\xi) = 2(r^2-1)(g-1)+2 \sum_{i=1}^{n}
		\sum_{j=1}^{\ell_i}
		\sum_{j' > j}
		m^i_j \, m^i_{j'}\,.
	\end{align}
	One can compute the Brauer group of $\mathcal{M}^{\boldsymbol{m,\alpha}}_{Higgs}(r, \xi)$ in a similar manner to Theorem \ref{Br of Para Conn}:
	\begin{theorem}\label{Br of Higgs}
		If $\boldsymbol{\alpha}$ be a generic system of weights as in Definition \ref{def:generic-weights}, then $\mathrm{Br}(\mathcal{M}^{\boldsymbol{m,\alpha}}_{Higgs}(r, \xi))$ is isomorphic to the cyclic group $\bb{Z}/\nu\bb{Z}$, where 
		$$\nu= \gcd(r,d,m_1^1,m_2^1,\cdots,m_1^{\ell_1},\cdots,m_{n}^{\ell_n})\,.$$
	\end{theorem}
	\begin{proof}
		Consider the moduli space $\mc{N}^{\boldsymbol{m,\alpha}}(r,\xi)$ of stable parabolic vector bundles on $X$ of rank $r$, determinant $\xi$, and parabolic data $(\boldsymbol{m,\alpha})$. Recall that $\mathrm{ParEnd}'(E_*)$ denote the space of parabolic endomorphisms of trace zero. Using this, the cotangent space of $\mc{N}^{\boldsymbol{m,\alpha}}(r,\xi)$ at $E_*$ is given by 
		$$(T_{E_{*}}\mc{N}^{\boldsymbol{m,\alpha}}(r,\xi))^{\vee} \simeq \mathrm{H}^1(X, \mathrm{ParEnd}'(E_{*}))^{\vee} \simeq \mathrm{H}^0(X,\mathrm{SParEnd}'(E_{*})\otimes \Omega_X^1(S)).$$
		In fact, the cotangent bundle $T^*\mc{N}^{\boldsymbol{m,\alpha}}(r,\xi)$ is embedded in $\mathcal{M}^{\boldsymbol{m,\alpha}}_{Higgs}(r, \xi)$ as a Zariski open subset whose complement 
		is of codimension at least two. Now, a similar reasoning as in the proof of Theorem \ref{Br of Para Conn} shows that
		$$\mathrm{Br}(\mathcal{M}^{\boldsymbol{m,\alpha}}_{Higgs}(r, \xi))\simeq \mathrm{Br}\left(T^*\mc{N}^{\boldsymbol{m,\alpha}}(r,\xi)\right) \simeq \mathrm{Br} (\mc{N}^{\boldsymbol{m,\alpha}}(r,\xi))\,.$$ 
		Since $\mathrm{Br}\left(\mathcal{N}^{\boldsymbol{m,\alpha}}(r,\xi)\right)\simeq \frac{\bb{Z}}{\nu\bb{Z}}$ \cite[Theorem 1.1]{BD}, our result follows.
	\end{proof}
	\section{Moduli of stable parabolic $\text{PGL}(r,\bb{C})$-connections}
	\subsection{Action of torsion line bundles on the moduli of parabolic connections}\hfill\\
	Let 
	\begin{align}\label{eqn:r-torsion-line-bundles}
		\Gamma := \{M\in\Pic(X)\,\mid\,M^{\otimes r}\simeq \mc{O}_X\}
	\end{align}
	denote the group of $r$-torsion line bundles on $X$. It is a finite abelian group of order $r^{2g}$. We aim to construct an action of $\Gamma$ on the moduli of parabolic stable connections $\mc{M}^{\boldsymbol{m,\alpha}}_{pc}(r,\xi)$ (see Definition \ref{def:parabolic-connections-moduli}), and study its fixed point locus.
	
	The trivial line bundle always admits a holomorphic connection, namely the differential $d: \mc{O}_X\longrightarrow \Omega^1_X$. This will be called as the \textit{trivial connection} on $\mc{O}_X$. The following result should be well-known to experts; we provide a proof here for the sake of completeness.
	\begin{lemma}\label{lem:connection-line-bundle}
		Let $L$ a holomorphic line bundle on $X$ satisfying $L^{\otimes r}\simeq \mc{O}_X$. There exists a unique holomorphic connection $D_L$ on $L$ such that the induced connection $D_L^{\otimes r}$ on $L^{\otimes r}$ coincides with the trivial connection $d$ on $\mc{O}_X$ under the identification. 
	\end{lemma} 
	\begin{proof}
		For any line bundle $M$, let $at(M)\in H^1(X,\Omega^1_X)$ denote its Atiyah class. It is well-known that $M$ admits a holomorphic connection iff $at(M)=0$ \cite[Theorem 2]{A}. 
		
		Choose an isomorphism $\varphi: L^{\otimes r}\xrightarrow{\simeq} \mc{O}_X$. The trivial connection $d$ on $\mc{O}_X$ can be transferred as a  holomorphic connection on $L^{\otimes r}$ via $\varphi$. Since the Atiyah class of a line bundle coincides with its first Chern class, we get
		\begin{align}
			at(L^{\otimes r}) = c_1(L^{\otimes r}) &=0\\
			\implies r\cdot c_1(L) &= 0.
		\end{align}
		It follows that $c_1(L) = at(L) =0$, as $c_1(L)\in H^2(X,\bb{Z})\simeq\bb{Z}$. Consequently, $L$ admits a holomorphic connection, say $D$. Consider $D^{\otimes r}$ as a connection on $\mc{O}_X$ by transferring it via $\varphi$. In other words, consider the following connection $\mc{O}_X$\,:
		$$D':=(\varphi\otimes \Id_{\Omega^1_X})\circ D^{\otimes r}\circ \varphi^{-1}\,.$$ 
		Now, since the space of holomorphic connections on $\mc{O}_X$ is an affine space modeled over $H^0(X,\Omega_X^1)$, it follows that there exists a holomorphic 1-form $\omega$ on $X$ satisfying
		\begin{align}
			D' -  \omega\,=\,d\,. 
		\end{align}
		Let $\omega':= \Id_{L^{\otimes r}}\otimes \omega \in H^0(X, \,\text{End}(L^{\otimes r})\otimes \Omega_{X}^1)$, and let $D_L := D-\dfrac{1}{r}\omega'$ be the resulting connection on $L$. It follows that $$D_L^{\otimes r}=D^{\otimes r}-r\cdot\dfrac{1}{r}\omega'=D^{\otimes r}-\omega'\,,$$
		and consequently,
		\begin{align}
			(\varphi\otimes\Id_{\Omega^1_{X}})\circ D_L^{\otimes r}\circ\varphi^{-1} = (\varphi\otimes\Id_{\Omega^1_{X}})\circ (D^{\otimes r}-\omega')\circ\varphi^{-1}\, = \,D'-\omega =d.
		\end{align}
		In other words, $D_L^{\otimes r}$ coincides with the trivial connection $d$ on $\mc{O}_X$ via $\varphi$.
		
		To show uniqueness, suppose $L$ admits another connection $D_0$ satisfying $$	(\varphi\otimes\Id_{\Omega^1_{X}})\circ D_0^{\otimes r}\circ \varphi^{-1} = d.$$ Now, $D_L = D_0+(\Id_{L})\otimes\eta$ for some holomorphic $1$-form $\eta$ on $X$. It follows that 
		$$D^{\otimes r}_L = D^{\otimes r}_0 + r\cdot (\Id_{L^{\otimes r}})\otimes\eta\,, $$
		which implies $(\varphi\otimes\Id_{\Omega^1_{X}})\circ D^{\otimes r}_L\circ\varphi^{-1} =(\varphi\otimes\Id_{\Omega^1_{X}})\circ D^{\otimes r}_0\circ\varphi^{-1}+r\cdot \eta$ on $\mc{O}_X$. Since $$(\varphi\otimes\Id_{\Omega^1_{X}})\circ D^{\otimes r}_L\circ\varphi^{-1} =(\varphi\otimes\Id_{\Omega^1_{X}})\circ D^{\otimes r}_0\circ\varphi^{-1}=d\,,$$
		we obtain $\eta=0$, showing uniqueness of $D_L$. Finally, if we choose another isomorphism $\varphi'$ between $L^{\otimes r}$ and $\mc{O}_X$, then $\varphi$ and $\varphi'$ differ by a nonzero scalar. It then follows from the above calculation that $D_L$ is independent of the choice of the isomorphism $\varphi$. 
	\end{proof}
	
	Using Lemma \ref{lem:connection-line-bundle}, one can define an action of $\Gamma$ on $\mc{M}^{\boldsymbol{m,\alpha}}_{pc}(r,\xi)$ as follows:  given $L\in \Gamma$ and a stable parabolic connection $(E_*,\nabla)$, define 
	$$L\cdot(E_*,\nabla):=(E_*\otimes L,\,\nabla\otimes D_L)\,,$$ 
	where $E_*\otimes L$ is given the obvious parabolic structure coming from $E_*$, and $(\nabla\otimes D_L)$ is defined as 
	$$(\nabla\otimes D) (s\otimes t) := \nabla(s)\otimes t + s\otimes D(t)$$ for locally defined sections $s$ of $E$ and $t$ of $L$. 
	
	Let us give some details on why this indeed defines an action. First, note that $$\det(E\otimes L)\simeq \det(E)\otimes L^{\otimes r}\simeq \xi.$$
	Also, it is straightforward to see that each square in the following diagram commutes:
	\begin{align*}
		\xymatrix{
			\det(E\otimes L) \ar[rrr]^(.4){\det(\nabla\otimes D_L)}\ar[d]_{\simeq} &&& \det(E\otimes L)\otimes \Omega_X^1(D)\ar[d]^{\simeq}\\
			\det(E)\otimes L^{\otimes r} \ar[rrr]^(.4){\det(\nabla)\otimes D_L^{\otimes r}}\ar[d]_{\Id_{\det(E)}\otimes\varphi} &&& \det(E)\otimes L^{\otimes r} \otimes\Omega^1_X(D)\ar[d]^{\Id_{\det(E)} \otimes\varphi\otimes\Id_{\Omega_X^1(D)}}\\
			\det(E) \ar[rrr]^(.4){\det(\nabla)} &&& \det(E)\otimes\Omega_X^1(D)
		}
	\end{align*}
	The outermost arrows imply that  $(\det(E\otimes L),\,\det(\nabla\otimes D_L))\simeq(\det(E),\,\det(\nabla))$ as logarithmic connections. Moreover, at each $x\in S$ we have  $\text{Res}_x(D_L)=0$ due to $D_L$ being holomorphic. It follows that
	\begin{align}\label{eqn:residue-tensor-product}
		\text{Res}_x(\nabla\otimes D_L) = \text{Res}_x(\nabla)\otimes \Id_{L_x}+\Id_{E_x}\otimes\text{Res}_x(D_L) = \text{Res}_x(\nabla)\otimes \Id_{L_x}\,.
	\end{align}
	Thus, $\text{Res}_x(\nabla \otimes D_L)$ preserves each subspace $E_j^x\otimes L_x$ in the flag of $E_x\otimes L_x$. Also,  $$\text{Res}_x(\nabla\otimes D_L)\left((E^x_j\otimes L_x)/(E^x_{j+1}\otimes L_x)\right) = \alpha^x_j\cdot\Id\,.$$ 
	Thus, $(E_*\otimes L, \,\nabla\otimes D_L)$ is a parabolic connection (see Definition \ref{def:parabolic-connections}). This way, we obtain an action of $\Gamma$ on $\mc{M}^{\boldsymbol{m,\alpha}}_{pc}(r,\xi)$.
	
	\subsection{Dimension estimate of fixed point locus}\hfill\\	
	Recall the group $\Gamma$ from \eqref{eqn:r-torsion-line-bundles}. Let $L\in\Gamma$ be non-trivial. If $m=ord(L)$, choose a nowhere-vanishing section $\sigma$ of $L$. Consider the morphism $g: L \longrightarrow L^{\otimes m},\,v\mapsto v^{\otimes m}$. The \textit{spectral curve} associated to $\sigma$ is defined as the closed subscheme of $L$ given by
	\begin{align}
		Y_L := g^{-1}(\sigma(X))\,.
	\end{align}  
	Let $\gamma:Y_L\longrightarrow X$ denote the restriction of the projection $L\longrightarrow X$ to $Y_L$. It is a finite \'etale Galois morphism  of degree $m$. The isomorphism class of the
	map $\gamma$ does not depend on the choice of the section $\sigma$.
	
	Given a parabolic data $(\boldsymbol{m,\alpha})$ of rank $r$ along $S$, a method of prescribing parabolic data of rank $\frac{r}{m}$ along $\gamma^{-1}(S)$ was described in \cite[\S~3]{BC}, which shall be briefly recalled below for convenience. For simplicity, consider the case where $S\,=\,\{p\}$ is a singleton set. Thus we have a
	system of weights $(\alpha_1^p\,<\ \alpha_2^p\ <\,\cdots\,<\  \alpha^p_{\ell(p)})$ at $p\in X$. 
	The group $\textnormal{Gal}(\gamma)=\{1, \,\mu,\,\mu^2,\,\cdots,\,\mu^{m-1}\}$ acts on $\gamma^{-1}(p)$ via multiplication from $\bb{C}^*$. Fix 
	an ordering on the points of $\gamma^{-1}(p)$, say 
	\begin{align}\label{eqn:fibre}
		\gamma^{-1}(p)\ =\ \{q_1,\,q_2,\,\cdots,\,q_m\},
	\end{align}
	such that $\mu^i$ acts on $\gamma^{-1}(p)$ as the cyclic permutation sending $q_j$ to $q_{j+i}$, where the 
	subscript $(j+i)$ is to be understood modulo $m$.
	Let $\boldsymbol{P}$ denote the collection of all nonempty subsets of $\left\{\alpha^p_1, \ \alpha^p_2, \ \cdots, \ \alpha^p_{\ell(p)}\right\}$ of cardinality \textit{at most} $\frac{r}{m}$. Define  
	\begin{equation}\label{eqn:set-of-subsets}
		\boldsymbol{T}\ \ :=\ \ \boldsymbol{P}\underset{m\ \textnormal{times}}{\times \cdots\times}\boldsymbol{P}.
	\end{equation}
	For each $\textbf{t}\,\in\, \textbf{T}$, we want to describe a procedure of associating a system of weights, as 
	well as a collection of systems of multiplicities along $\gamma^{-1}(p)\,=\,\{q_1,\ q_2,\ \cdots,\ q_m\}$. 
	For each $\textbf{t}\,\in \,\textbf{T}$, adopt the notation
	\begin{equation}\label{eqn:set-of-subsets-2}
		\textbf{t} \ =\ \left(\lambda_1(\textbf{t}),\ \lambda_2(\textbf{t}), \ \cdots ,\
		\lambda_d(\textbf{t})\right),
	\end{equation}
	where each $\lambda_j(\textbf{t})\neq \emptyset$ by assumption. Clearly each $\lambda_j(\textbf{t})$ can be 
	arranged into an increasing sequence. Take $\lambda_j(\textbf{t})$ as the set of weights at $q_j$ for 
	each $1\,\leq\, j\,\leq\, d$. Thus, each $\textbf{t}\,\in\, \boldsymbol{T}$ prescribes a system of weights on 
	$\gamma^{-1}(p)$, which we shall denote by
	$\boldsymbol{\alpha}_{\textbf{t}}$.
	
	Next, for each fixed $\textbf{t}\,\in\, \textbf{T}$, denote by
	\begin{align}\label{eqn:collection-of-sequences}
		A_{\textbf{t}}
	\end{align}
	as the collection of all matrices of size $(m\times \ell(p))$ with \textit{non-negative integer} entries, written in the form
	\begin{align}\label{eqn:matrix}
		\begin{pmatrix}
			n^{q_1}_1 & n^{q_1}_2 & \cdots & \cdots& n^{q_1}_{\ell(p)} \\
			n^{q_2}_1 & n^{q_2}_2 & \cdots & \cdots& n^{q_2}_{\ell(p)} \\
			\vdots & \vdots & \vdots& \vdots &\vdots\\
			n^{q_m}_1 & n^{q_m}_2 & \cdots & \cdots& n^{q_m}_{\ell(p)} \\
		\end{pmatrix}
	\end{align}
	which further satisfy the following two conditions:
	\begin{enumerate}[(a)]
		\item \label{condition-on-multiplicity-1} The numbers in the $j$-th row, namely the
		sequence 
		$$\left(n^{q_j}_1,\, n^{q_j}_2,\,\cdots,\, n^{q_j}_{\ell(p)}\right)$$
		must satisfy the condition that for every  $1\,\leq\, k\,\leq\,\ell(p)$,
		$$(n^{q_j}_k\ =\ 0)\ \ \iff\ \ (\alpha^{p}_k\ \notin\ \lambda_j(\textbf{t})) \,\,\,(\text{see}\ \eqref{eqn:set-of-subsets-2});
		$$ 
		note that this makes sense because
		$\lambda_j(\textbf{t})\,\subset\,\{\alpha^{p}_1,\ \alpha^{p}_2,\ \cdots,\ \alpha^{p}_{\ell(p)}\}$
		by construction.
		
		\item \label{condition-on-multiplicity-2} 
			the entries in each row of the matrix must add up to $\frac{r}{m}$, while the entries
			of the $k$-th column of it must add up to $m^p_k$ for each $k\,\in\,[1,\,\ell(p)]$. In other words,
			\[\left(\sum_{k=1}^{\ell(p)} n^{q_j}_k \ = \ \frac{r}{m}\ \ \ \forall\ \ j\,\in\,[1,\,m]\right)\,\,\,
			\text{and}\,\,\,\left(\sum_{j=1}^{m} n^{q_j}_k \ = \ m^p_k\ \ \ \forall\ \ k\,\in\,[1,\,\ell(p)]\right).\]
		\end{enumerate}  
		For each $\textbf{t}\,\in\, \textbf{T}$ consider the system of weights
		$\boldsymbol{\alpha}_{\textbf{t}}$ along $\gamma^{-1}(p)$ as described just before
		\eqref{eqn:set-of-subsets}. It was shown in \cite[Lemma 3.2]{BC} that the collection $A_{\textbf{t}}$ gives rise to a system of multiplicities along $\gamma^{-1}(p)$, denoted by $\boldsymbol{n}$. Let $d\,=\,\deg(\xi)$, and suppose
		\begin{align*}
			\mc{M}^{\textbf{n},\boldsymbol{\alpha}_{\textbf{t}}}_{Y,pc}\left(\frac{r}{m},d\right)
		\end{align*}
		denote the moduli space of parabolic stable connections on $Y$ along $\gamma^{-1}(S)$ of rank $\frac{r}{m}$ and
		degree $d\,=\,\deg(\xi)$, and having parabolic data
		$(\boldsymbol{n},\,\boldsymbol{\alpha}_{\textbf{t}})$ along $\gamma^{-1}(S)$. Consider the following space for each
		$\textbf{t}\,\in\, \textbf{T}$:
		\begin{equation}\label{eqn:moduli}
			\mathcal{N}^\textbf{t}_L\ :=\ \left\{(F_*,\,\nabla)\,\in\, \coprod_{\boldsymbol{n}\,\in\,
				A_{\textbf{t}}}\mc{M}_{Y,pc}^{\boldsymbol{n},\boldsymbol{\alpha}_{\textbf{t}}}\left(\frac{r}{m},d\right)\
			\big\vert\ (\det(\gamma_*F),\,\det(\nabla))\,\simeq\, (\xi,\nabla_{\xi})\right\}.
		\end{equation}
		
		\begin{proposition}\label{prop:parabolic-fixed-point}
			Let $(\boldsymbol{m,\alpha})$ be a parabolic data  along $S$ such that $\boldsymbol{\alpha}$ is a generic system of weights as in Definition \ref{def:generic-weights}. Let $L$ be a nontrivial $r$-torsion line bundle on $X$. The parabolic push-forward along the spectral curve $\gamma:Y_L\longrightarrow X$ gives rise to a surjective morphism $$f: \left(\coprod_{\boldsymbol{t}\,\in\,\textbf{T}}\mc{N}^{\boldsymbol{t}}_L\right) \longrightarrow \mc{M}^{\boldsymbol{m,\alpha}}_{pc}(r,\xi)^L\,.$$ 
		\end{proposition}
		\begin{proof}
			The map $f$ sends a stable parabolic connection $(E_*,\nabla)$ to $\left(\gamma_*E_*,\,\gamma_*\nabla\right)$. The image $\left(\gamma_*E_*,\,\gamma_*\nabla\right)$ is parabolic semistable \cite[Proposition 6.4]{BA}. Since the system of weights $\boldsymbol{\alpha}$ is chosen to be generic, $\left(\gamma_*E_*,\,\gamma_*\nabla\right)$ is also parabolic stable. Moreover, by \cite[Proposition 3.2]{BH}, $\left(\gamma_*E_*,\,\gamma_*\nabla\right)$ is a fixed point for the action of $L\in \Gamma$. This shows that the map $f$ is well-defined. 
			
			To check surjectivity of $f$, consider a fixed point $(E_*,\nabla)$ under the action of $L$. Consider the underlying vector bundle $E$ with the logarithmic connection $\nabla$ along $S$. By \cite[Theorem 3.1]{BH}, there exists a vector bundle $F$ of rank $\frac{r}{m}$ over $Y$ together with a logarithmic connection $\nabla_F$ along $\gamma^{-1}(S)$ satisfying $(\gamma_*F,\gamma_*\nabla_F)\xrightarrow{\simeq}(E,\nabla)$. 
			
			Let us briefly recall the construction of $F$. Since $(E_*,\nabla)$ is a fixed point under the action of $L$, there exists an isomorphism $\varphi:E\xrightarrow{\,\,\simeq\,\,} E\otimes L$. The line bundle $\gamma^*L$ admits a tautological trivialization $\theta:\gamma^*L\xrightarrow{\,\,\simeq\,\,}\mc{O}_{Y_{L}}$. Composing the pullback $\gamma^*\varphi$  with $\theta$ gives rise to an automorphism $\phi: \gamma^*E \xrightarrow{\,\,\simeq\,\,}\gamma^*E$. Then $\gamma^*E$ splits as a direct sum of eigenspace sub-bundles, and $F$ is obtained as one eigenspace sub-bundle of $\gamma^*E$ under $\phi$.
			Moreover, for any $x\in S$ and $y\in \gamma^{-1}(S)$, there exists a filtration of $F_y$ by eigenspaces of $\phi_y$  \cite[Proposition 3.3]{BC}. Assigning a system of weights appropriately to these filtrations induces a parabolic structure on $F_*$ along $\gamma^{-1}(S)$, such that the parabolic pushforward $\gamma_*(F_*)$ coincides with $E_*$ (see  \cite[Proposition 3.3]{BC} for more details). 
			Moreover, $F$ is preserved by $\gamma^*\nabla$ by construction, and the logarithmic connection $\nabla_F$ is nothing but the restriction of $\gamma^*\nabla$ to $F$ \cite[Theorem 3.1]{BH}.
			
			We claim that $\nabla_F$ is a parabolic connection with respect to the induced parabolic structure on $F$ as above.  First of all, since $\gamma$ is \'etale and thus unramified, we have $\text{Res}_y(\gamma^*\nabla)=\text{Res}_x(\nabla)$ under the identification $(\gamma^*E)_y = E_x$ \cite[\S~5.2]{BA}. Consequently, $\text{Res}_y(\gamma^*\nabla)$ is a semisimple endomorphism of $(\gamma^*E)_y$.  Since $\nabla_F$ is the restriction of $\gamma^*\nabla$ to $F$, it immediately follows that $\text{Res}_y(\nabla_F)$ is a semisimple endomorphism of $F_y$. Moreover, since $\text{Res}_y(\gamma^*D_L) =0$ due to $D_L$ being holomorphic, we have 
			\[\text{Res}_y\left(\gamma^*\nabla\otimes \gamma^*D_L\right) = \text{Res}_y(\gamma^*\nabla)\otimes \Id_{(\gamma^*L)_y}.\]
			From this, and the description of the automorphism $\varphi$ of $\gamma^*E$ given above, it can be easily shown that $\text{Res}_y(\gamma^*\nabla)$ commutes with $\phi_y$. Thus, $\text{Res}_y(\gamma^*\nabla)$ preserves the parabolic filtration of $F_y$ given by eigenspaces of $\phi_y$ constructed above. Finally, it is not difficult to see from the construction of the weighted filtration of $F_y$ in \cite[Proposition 3.3]{BC} that $\nabla_F$ acts as a scalar given by the appropriate weight on the successive quotients of the filtration, arising from the fact $\text{Res}_y(\nabla_F)=\text{Res}_x(\nabla)$ and that $\nabla$ is a parabolic connection. It follows that $\nabla_F$ satisfies all conditions of Definition \ref{def:parabolic-connections}. Thus, $(F_*,\, \nabla_F)$ is a parabolic connection on $Y$ of rank $\frac{r}{m}$ satisfying $$(\gamma_*F_*,\,\gamma_*\nabla_F)\xrightarrow{\,\,\simeq\,\,}(E_*,\nabla)\,.$$
			Moreover, by \cite[Theorem 1.1 (2)]{BA}, $(F_*,\nabla_F)$ is parabolic semistable, and from the parabolic stability of $(E_*,\nabla)$ it follows that $(F_*,\nabla_F)$ must be parabolic stable as well, since otherwise the existence of a sub-bundle $F'\subset F$ preserved under $\nabla_F$ satisfying $p\mu(F'_*)\geq p\mu(F_*)$ would imply that $\gamma_*F'_*\subset \gamma_*F = E$ is a sub-bundle preserved under $\nabla$ which satisfies $p\mu(\gamma_*{F'_*})\geq p\mu(\gamma_*{F_*})=p\mu(E_*)$. This contradicts the parabolic stability of $(E_*,\nabla)$. Thus $f$ is surjective.
		\end{proof}
		\begin{remark}
			If the cardinality of $S$ is greater than one, we repeat this
			construction for each of the parabolic points in $S$. Say, for example, $S\,=\,\{p_1,\ p_2,\ \cdots,\ p_n\}$;
			then for each $1\,\leq\, i\,\leq\, n$, define the set $\boldsymbol{T}_i$ analogous to $\boldsymbol{T}$ in
			\eqref{eqn:set-of-subsets}, and replace $\boldsymbol{T}$ by $\boldsymbol{T}_1\times\boldsymbol{T}_2
			\times\cdots\boldsymbol{T}_n$. Then, a similar argument as in Proposition
			\ref{prop:parabolic-fixed-point} will prove the same result, using more complicated notations.
		\end{remark}
		\begin{corollary}\label{cor:codimension}
			Let $\boldsymbol{\alpha}$ be a generic system of weights. The codimension of the closed subscheme 
			\begin{align}
				Z_{\boldsymbol{\alpha}} := \bigcup_{L\in\Gamma\setminus\{\mc{O}_X\}} \mc{M}^{\boldsymbol{m,\alpha}}_{pc}(r,\xi)^L \,\subset \, \mc{M}^{\boldsymbol{m,\alpha}}_{pc}(r,\xi)
			\end{align}
			is at least three.
		\end{corollary}
		\begin{proof}
			The calculations for the codimension estimate are very similar to the codimension estimate considered in \cite[Corollary 3.5]{BC}. The heuristic principle here is the fact that the dimension of a space consisting of parabolic connections with a certain property like stability is \textit{twice} the dimension of the space consisting of ordinary parabolic vector bundles having the same property.  (For example, we have $\dim \mc{M}^{\boldsymbol{m,\alpha}}_{pc}(r,\xi) = 2\dim \mc{N}^{\boldsymbol{m,\alpha}}(r,\xi)$\,; see \eqref{eqn:dimension-parabolic-moduli}). 
			Apart from this distinction, all calculations remain exactly similar as in \cite[Corollary 3.5]{BC}. We omit the details.
		\end{proof}
		
		\section{Brauer group of the moduli stack of stable parabolic $\text{PGL}(r,\bb{C})$-connections}\label{section:brauer-group-moduli-stack}				
		\begin{definition}
			A \textit{stable parabolic }$\textnormal{PGL}(r,\bb{C})$-\textit{connection} is an equivalence class of stable parabolic connections, where two parabolic stable connections $(E_*,\,\nabla)$ and $(E'_*,\,\nabla')$ 
			are considered equivalent if there exists an $r$-torsion line bundle $L$ together with an isomorphism of parabolic connections $(E'_*,\,\nabla')\,\xrightarrow{\,\,\simeq\,\,}\, (E_*\otimes L,\,\nabla\otimes D_L)$ , where $D_L$ is the unique connection on $L$ constructed in Lemma \ref{lem:connection-line-bundle}.
		\end{definition}
		
		\begin{definition}\label{def:moduli-pgl-connections}
			Fix a line bundle $\xi$ on $X$. Suppose $i\in[0,r-1]$ satisfies $i\equiv d\,\,(\text{mod}\,r)$, where $d=\deg(\xi)$. The \textit{coarse moduli space of stable parabolic $\text{PGL}(r,\bb{C})$-connections on} $X$ having parabolic data $(\boldsymbol{m,\alpha})$ is defined as the quotient space 
			$$N_{pc}^{\boldsymbol{m,\alpha}}(r,i) :=\mc{M}_{pc}^{\boldsymbol{m,\alpha}}(r,\xi)/\Gamma\,.$$
		\end{definition}
		Also, denote by $\mf{N}^{\boldsymbol{m,\alpha}}_{pc}(r,i)$ the moduli \textit{stack} of stable parabolic
		$\text{PGL}(r,\bb{C})$-connections on $X$ of topological type $i$ and systems of multiplicities and
		weights $(\boldsymbol{m,\,\alpha})$.
		
		\begin{theorem}\label{thm:brauer-group-of-moduli-stack}
			Let $\boldsymbol{\alpha}$ be a generic system of weights for connections, as in Definition \ref{def:generic-weights-connections}. Let $N_{pc}^{\boldsymbol{m,\alpha}}(r,i)^{sm}$ denote the smooth locus of $N_{pc}^{\boldsymbol{m,\alpha}}(r,i)$. We have
			$$\textnormal{Br}\left(\mf{N}_{pc}^{\boldsymbol{m,\alpha}}(r,i)\right)\,\ \simeq\,\
			\textnormal{Br}\left(N_{pc}^{\boldsymbol{m,\alpha}}(r,i)^{sm}\right).$$
		\end{theorem}
		\begin{proof}
			The idea of the proof is similar to  \cite[Theorem 4.2]{BC}, so we omit some details. 	
			Let $Z_{\boldsymbol{\alpha}}$ be as in Corollary \ref{cor:codimension}, and denote
			$U\,:=\,N_{pc}^{\boldsymbol{m,\alpha}}(r,i)\setminus Z_{\boldsymbol{\alpha}}$. Then $\Gamma$ acts freely on $U$. Thus the quotient variety $U/\Gamma$ is smooth. Since taking quotient by the finite group $\Gamma$ is a finite morphism and hence codimension-preserving, it follows from Corollary \ref{cor:codimension} that the complement of $U/\Gamma$ in $\left(N_{pc}^{\boldsymbol{m,\alpha}}(r,i)\right)^{sm}$ is of codimension at least $3$. For a similar reason, the complement of the open sub-stack $[U/\Gamma]$ in $\mf{N}_{pc}^{\boldsymbol{m,\alpha}}(r,i)$ is of codimension at least $3$ as well. 
			As $\mf{N}_{pc}^{\boldsymbol{m,\alpha}}(r,i)$ is a Deligne-Mumford stack, it follows that
			\begin{equation*}
				\Br\left([U/\Gamma]\right)\,\ \simeq\, \ \Br\left(\mf{N}^{\boldsymbol{m,\alpha}}_{pc}(r,i)\right)\,\,\,\text{(see \cite[Proposition 4.2]{BCD23}).}
			\end{equation*}
			Since $\Gamma$ acts freely on $U$, we have $U/\Gamma \simeq [U/\Gamma]$. Consequently,
			$\Br\left([U/\Gamma]\right)\ \simeq\
			\Br\left(U/\Gamma\right).$
			Also, as $\Gamma$ acts freely on $U$ and $U$ is smooth, the quotient $U/\Gamma$ is also
			smooth. The complement of $U/\Gamma$ in $\left(\mc{N}^{\boldsymbol{m,\alpha}}_{pc}(r,i)\right)^{sm}$
			clearly has codimension at least $3$ as well. Thus the above isomorphisms imply
			$\Br\left(\mf{N}^{\boldsymbol{m,\alpha}}_{pc}(r,i)\right)\,\simeq\,
			\Br\left(U/\Gamma\right)\, \simeq\,  \Br\left(\mc{N}^{\boldsymbol{m,\alpha}}_{pc}(r,i)^{sm}\right),$ 
			where the last isomorphism follows from cohomological purity \cite[Theorem 1.1]{C}. This completes the proof.
		\end{proof}
		\section*{Acknowledgements}
		The first-named author is supported by the SERB project number SPON/SERB/62041. The second-named author is supported by the DST-INSPIRE Faculty Fellowship (Research Grant No.: DST/INSPIRE/04/2024/001521), Ministry of Science and Technology, Government of India.

	\end{document}